\documentclass[11pt,english]{article}

\usepackage[T1]{fontenc}
\usepackage[latin9]{inputenc}
\usepackage{geometry}
\usepackage{babel}
\usepackage{amsmath}
\usepackage{amsthm}
\usepackage{amssymb}
\usepackage[unicode=true,pdfusetitle,
 bookmarks=true,bookmarksnumbered=false,bookmarksopen=false,
 breaklinks=false,pdfborder={0 0 0},pdfborderstyle={},backref=false,colorlinks=false]
 {hyperref}
\usepackage{breakurl}
\makeatletter

\usepackage[nameinlink,capitalise,noabbrev]{cleveref}

\hypersetup{%
    bookmarksnumbered, bookmarksopen=true, bookmarksopenlevel=1,%
}

\theoremstyle{plain}
\newtheorem{thm}{Theorem}[section]
\crefname{thm}{Theorem}{Theorems}
\theoremstyle{plain}
\newtheorem{lem}[thm]{Lemma}
\crefname{lem}{Lemma}{Lemmas}
\theoremstyle{plain}

\theoremstyle{plain}
\newtheorem*{claim*}{Claim}
\crefname{claim}{Claim}{Claims}
\theoremstyle{definition}
\newtheorem{defn}[thm]{Definition}
\theoremstyle{plain}
\newtheorem{conjecture}[thm]{Conjecture}
\theoremstyle{plain}

\theoremstyle{definition}

\theoremstyle{definition}
\newtheorem{rem}[thm]{Remark}
\theoremstyle{plain}
\newtheorem*{rem*}{Remark}
\theoremstyle{plain}
\newtheorem{claim}[thm]{Lemma}
\crefname{claim}{Lemma}{Lemmas}

\crefformat{equation}{#2(#1)#3}
\crefname{appsec}{Appendix}{Appendices}

\date{}

\let\originalleft\left
\let\originalright\right
\renewcommand{\left}{\mathopen{}\mathclose\bgroup\originalleft}
\renewcommand{\right}{\aftergroup\egroup\originalright}

\usepackage{verbatim}

\makeatletter
\renewcommand*{\UrlTildeSpecial}{%
  \do\~{%
    \mbox{%
      \fontfamily{ptm}\selectfont
      \textasciitilde
    }%
  }%
}%
\let\Url@force@Tilde\UrlTildeSpecial
\makeatother

\usepackage{enumitem}

\makeatother

\begin{document}

\title{Halfway to Rota's basis conjecture}

\author{Matija Buci\'c\thanks{Department of Mathematics, ETH, Z\"urich, Switzerland. Email: \href{mailto:matija.bucic@math.ethz.ch} {\nolinkurl{matija.bucic@math.ethz.ch}}.}\and
Matthew Kwan\thanks{Department of Mathematics, Stanford University, Stanford, CA 94305. Email: \href{mailto:mattkwan@stanford.edu} {\nolinkurl{mattkwan@stanford.edu}}. Research supported in part by SNSF project 178493.}\and
Alexey Pokrovskiy\thanks{Department of Economics, Mathematics and Statistics, Birkbeck, University of London. Email:
\href{mailto:Dr.Alexey.Pokrovskiy@gmail.com} {\nolinkurl{Dr.Alexey.Pokrovskiy@gmail.com}}.}\and Benny Sudakov\thanks{Department of Mathematics, ETH, Z\"urich, Switzerland. Email:
\href{mailto:benjamin.sudakov@math.ethz.ch} {\nolinkurl{benjamin.sudakov@math.ethz.ch}}.
Research supported in part by SNSF grant 200021-175573.}}

\maketitle
\global\long\def\E{\mathbb{E}}
\global\long\def\Var{\operatorname{Var}}
\global\long\def\S{\mathcal{S}}

\begin{abstract}
In 1989, Rota made the following conjecture. Given $n$
bases $B_{1},\dots,B_{n}$ in an $n$-dimensional vector space $V$,
one can always find $n$ disjoint bases of $V$, each containing exactly
one element from each $B_{i}$ (we call such bases \emph{transversal
bases}). Rota's basis conjecture remains wide open despite its apparent
simplicity and the efforts of many researchers (for example, the conjecture
was recently the subject of the collaborative ``Polymath'' project).
In this paper we prove that one can always find $\left(1/2-o\left(1\right)\right)n$
disjoint transversal bases, improving on the previous best bound of
$\Omega\left(n/\log n\right)$. Our results also apply to the more
general setting of matroids.
\end{abstract}

\section{Introduction}

Given bases $B_{1},\dots,B_{n}$ in an $n$-dimensional vector space
$V$, a \emph{transversal basis }is a basis of $V$ containing a single
distinguished vector from each of $B_{1},\dots,B_{n}$. Two transversal
bases are said to be \emph{disjoint} if their distinguished vectors
from $B_{i}$ are distinct, for each $i$ (here ``distinguished'' means that two copies of the same vector appearing in two $B_i$s are considered distinct). In 1989, Rota conjectured
(see \cite[Conjecture~4]{HR94}) that for any vector space $V$ over
a characteristic-zero field, and any choice of $B_{1},\dots,B_{n}$,
one can always find $n$ pairwise disjoint transversal bases.

Despite the apparent simplicity of this conjecture, it remains wide
open, and has surprising connections to apparently unrelated subjects.
Specifically, it was discovered by Huang and Rota \cite{HR94} that
there are implications between Rota's basis conjecture, the Alon--Tarsi
conjecture \cite{AT92} concerning enumeration of even and odd Latin
squares, and a certain conjecture concerning the supersymmetric bracket
algebra.

Rota also observed that an analogous conjecture could be made in the
much more general setting of \emph{matroids}, which are objects that
abstract the combinatorial properties of linear independence in vector
spaces. Specifically, a finite matroid $M=\left(E,\mathcal{I}\right)$
consists of a finite ground set $E$ (whose elements may be thought
of as vectors in a vector space), and a collection $\mathcal{I}$
of subsets of $E$, called independent sets. The defining properties
of a matroid are that:
\begin{itemize}

\item{the empty set is independent (that is, $\emptyset\in\mathcal{I}$);}

\item{subsets of independent sets are independent (that is, if $A'\subseteq A\subseteq E$
and $A\in\mathcal{I}$, then $A'\in\mathcal{I}$);}

\item{if $A$ and $B$ are independent sets, and $\left|A\right|>\left|B\right|$,
then an independent set can be constructed by adding an element of
$A$ to $B$ (that is, there is $a\in A\backslash B$ such that $B\cup\left\{ a\right\} \in\mathcal{I}$).
This final property is called the \emph{augmentation property}.}

\end{itemize}

Observe that any finite set of elements in a vector space (over any
field) naturally gives rise to a matroid, though not all matroids
arise this way. A \emph{basis }in a matroid $M$ is a maximal independent
set. By the augmentation property, all bases have the same size, and
this common size is called the \emph{rank }of $M$. The definition
of a transversal basis generalises in the obvious way to matroids,
and the natural matroid generalisation of Rota's basis conjecture
is that for any rank-$n$ matroid and any bases $B_{1},\dots,B_{n}$,
there are $n$ disjoint transversal bases.

Although Rota's basis conjecture remains open, various special cases
have been proved. Several of these have come from the connection between
Rota's basis conjecture and the Alon--Tarsi conjecture,
which has since been simplified by Onn \cite{Onn97}. Specifically,
due to work by Drisko \cite{Dri97} and Glynn \cite{Gly10} on the
Alon--Tarsi conjecture, Rota's original conjecture for vector
spaces over a characteristic-zero field is now known to be true whenever
the dimension $n$ is of the form $p\pm1$, for $p$ a prime. Wild
\cite{Wil94} proved Rota's basis conjecture for so-called ``strongly
base-orderable'' matroids, and used this to prove the conjecture
for certain classes of matroids arising from graphs. Geelen and Humphries
proved the conjecture for ``paving'' matroids \cite{GH06}, and
Cheung \cite{Che12} computationally proved that the conjecture holds
for matroids of rank at most 4.

Various authors have also proposed variations and weakenings of Rota's
basis conjecture. For example, Aharoni and Berger \cite{AB06} showed
that in any matroid one can cover the set of all the elements in $B_{1},\dots,B_{n}$
by at most $2n$ ``partial'' transversals, and Bollen and Draisma
\cite{BD15} considered an ``online'' version of Rota's basis conjecture,
where the bases $B_{i}$ are revealed one-by-one. In 2017, Rota's
basis conjecture received renewed interest when it was chosen as the
twelfth ``Polymath'' project, in which amateur and professional
mathematicians from around the world collaborated on the problem.
Some of the fruits of the project were a small improvement to Aharoni
and Berger's theorem, and improved understanding of the online version
of Rota's basis conjecture \cite{Pol17}. See \cite{polyproposal}
for Timothy Chow's proposal of the project, see \cite{poly1,poly2,poly3}
for blog posts where much of the discussion took place, and see \cite{polywiki}
for the Polymath wiki summarising most of what is known about Rota's
basis conjecture.

One particularly natural direction to attack Rota's problem is to
try to find lower bounds on the number of disjoint transversal bases. Rota's basis conjecture
asks for $n$ disjoint transversal bases, but it is not completely
obvious that even two disjoint transversal bases must exist! Wild
\cite{Wil94} proved some lower bounds for certain matroids arising
from graphs, but the first nontrivial bound for general matroids was
by Geelen and Webb \cite{GW07}, who used a generalisation of Hall's
theorem due to Rado \cite{Rad42} to prove that there must be $\Omega\left(\sqrt{n}\right)$
disjoint transversal bases. Recently, this was improved by Dong and
Geelen \cite{DG18}, who used a beautiful probabilistic argument to
prove the existence of $\Omega\left(n/\log n\right)$ disjoint transversal
bases. In this paper we improve this substantially and obtain the
first linear bound.
\begin{thm}
\label{thm:new}For any $\varepsilon>0$, the following holds for
sufficiently large $n$. Given bases $B_{1},\dots,B_{n}$ of a rank-$n$
matroid, there are at least $\left(1/2-\varepsilon\right)n$ disjoint
transversal bases.
\end{thm}

Of course, since matroids generalise vector spaces, this also implies
the same result for bases in an $n$-dimensional vector space. We
also remark that for the weaker fact that there exist $\Omega\left(n\right)$
disjoint transversal bases, our methods give a simpler proof;
see \cref{rem:linear}.

In contrast to the previous work by Dong, Geelen and Webb, our approach
is to show how to build a collection of transversal bases in an iterative
fashion (reminiscent of augmenting path arguments in matching problems).
It is tempting to imagine a future path to Rota's basis conjecture
(at least in the case of vector spaces) using such an approach: by
improving on our arguments, perhaps introducing some randomness, it
might be possible to iteratively build a collection of $\left(1-o\left(1\right)\right)n$
transversal bases, and then it might be possible to use some sort
of ``template'' or ``absorber'' structure to finish the job. This
was precisely the approach taken in Keevash's celebrated proof of
the existence of designs \cite{Kee14}. Actually, it has been observed
by participants of the Polymath project (see \cite{poly1}) that Rota's
basis conjecture and the existence of designs conjecture both seem
to fall into a common category of problems which are not quite ``structured''
enough for purely algebraic methods, but too structured for probabilistic
methods.

\vspace{0.30cm}

\noindent{\bf Notation.} We will frequently want to denote the result of adding and removing single elements from a set. For a set $S$ and some $x\notin S$, $y\in S$, we write $S+x$ to mean $S\cup \{x\}$, and we write $S-y$ to mean $S\setminus \{y\}$.

\section{Finding many disjoint transversal bases}

In this section we prove \cref{thm:new}. It is convenient to think
of $B_{1},\dots,B_{n}$ as ``colour classes''.
\begin{defn}
Let $U=\left\{ \left(x,c\right):x\in B_{c},1\le c\le n\right\} $ be the set of
all coloured elements that appear in one of $B_{1},\dots,B_{n}$.
For $S\subseteq U$, let $\pi\left(S\right)=\left\{ x:\left(x,c\right)\in S\;\text{for some }c\right\} $
be its set of matroid elements. We say that a subset of elements of $U$
is a \emph{rainbow independent set} (RIS for short) if all its matroid
elements are distinct and form an independent set, and all their colours
are distinct.
\end{defn}

Note that an RIS with size $n$ corresponds to a transversal basis.
We remark that RISs are sometimes also known as \emph{partial transversals}.
Note that two transversal bases are disjoint if and only if their
corresponding RISs are disjoint as subsets of $U$.

Let $f=\left(1-\varepsilon\right)n/2$. The basic idea is to start
with a collection of $f$ empty RISs (which are trivially disjoint),
and iteratively enlarge the RISs in this collection, maintaining disjointness,
until we have many disjoint transversal bases.

Let $\S$ be a collection of $f$ disjoint RISs. We define the \emph{volume
}$\sum_{S\in\S}\left|S\right|$ of $\S$ to be the total number of
elements in the RISs in $\S$. We will show how to modify $\S$ to
increase its volume. We let $F=\bigcup_{S \in \S} S$ be the set of all currently used elements. One should think of $F$ as being the set of all elements which we cannot add to any $S\in\S$
without violating the disjointness of RISs in $\S$. 

We stress that in the following two subsections we fix a collection $\S$ and define $F$ as above. All our definitions and claims are with respect to these $F$ and $\S$. We will show that under certain conditions the size of $\S$ can be increased, at which point one needs to restart the argument from the beginning with a new $\S$ (and a new $F$). This is made precise in \Cref{subsec:increasing}. 

\begin{rem*}
We remark that it is actually possible to reduce to the case where each $B_c$ is disjoint, by making duplicate copies of all elements that appear in multiple $B_c$. So, instead of working with the universe $U$ of element/colour pairs, one can alternatively think of $U$ as being a collection of $n^2$ different matroid elements (each of which has a colour associated with it).
\end{rem*}

\subsection{\label{subsec:simple-swaps}Simple swaps}

Our objective is to increase the volume of $\S$. If an RIS $S\in\S$
is missing a colour $c$ and there is $x\in B_{c}$ independent to
the elements of $S$, such that $\left(x,c\right) \notin F$, then we can add $\left(x,c\right)$ to $S$
to create a larger RIS, increasing the volume of $\S$. We will want
much more freedom than this: we also want to consider those elements
that can be added to $S$ after making a small change to $S$. This
motivates the following definition.
\begin{defn}
\label{def:addable}Consider an RIS $S$ and a colour $b$ that does
not appear in $S$. Say an element $\left(x,c\right)\in U$ (possibly $(x,c)\in F$) is $\left(S,b\right)$-\emph{addable
}if either
\begin{itemize}
\item $S+\left(x,c\right)$ is an RIS, or;
\item There is $\left(x',c\right)\in S$ and $\left(y,b\right)\notin F$
such that $S-\left(x',c\right)+\left(y,b\right)+\left(x,c\right)$
is an RIS.
\end{itemize}
In the second case we say that $y$ is a \emph{witness} for the $\left(S,b\right)$-addability of $\left(x,c\right)$. For $(x',c) \in S$ and $(y,b) \notin F$ when $S-\left(x',c\right)+\left(y,b\right)$ is an RIS we say it
is the result of applying a \emph{simple swap }to $S$.
\end{defn}

If for some RIS $S\in\S$ missing a colour $b$ there is an $\left(S,b\right)$-addable
element $\left(x,c\right)\notin F$,
then we can increase the volume of $\S$ by adding $\left(x,c\right)$
to $S$, possibly after applying a simple swap to $S$. Note that we do not require $S\in \S$ for the definition of $(S,b)$-addability, though in practice we will only ever consider $S$ that are either in $\S$ or slight modifications of RISs in $\S$.

Our next objective is to show that for any $S$ missing a colour $b$,
either there is an $\left(S,b\right)$-addable element that is not
in $F$ (which would allow us to increase the volume of $\S$, as
above), or else there are \emph{many }$\left(S,b\right)$-addable
elements (which must therefore be in $F$). Although this will not
allow us to immediately increase the volume of $\S$, it will allow
us to transfer an element to $S$ from some other $S'\in\S$, and
this freedom to perform local modifications will be very useful.

Towards this end, we study which elements of $S$ can be used in a
simple swap.
\begin{defn}
Consider an RIS $S$ and consider a colour $b$ that does not appear
on $S$. We say that a colour $c$ appearing on $S$ is \emph{$\left(S,b\right)$-swappable}
if there is a simple swap yielding an RIS $S+\left(y,b\right)-\left(x',c\right)$,
with $\left(y,b\right)\notin F$ and $\left(x',c\right)\in S$. (For $S+\left(y,b\right)-\left(x',c\right)$ to be an RIS, we just need $\pi(S)+y-x'$ to be an independent set in our matroid.) We say that $y$ is a witness for
the $\left(S,b\right)$-swappability of $c$.
\end{defn}

(Basically, a colour is $\left(S,b\right)$-swappable if
we can replace it with a $b$-coloured element which is not in
$F$). For a colour $c$ we denote by $F_{c}=\left\{ x\in B_{c}:\left(x,c\right)\in F\right\}$ the set of matroid elements which appear in $\S$ with colour $c$.
\begin{claim}
\label{claim:many-good}For a nonempty RIS $S$ and a colour $b$ not
appearing in $S$, either there is an $\left(S,b\right)$-addable element $(y,b)\notin F$ or there are at least $n-\left|F_{b}\right|$ colours
which are $\left(S,b\right)$-swappable.
\end{claim}

\begin{proof}
For the purpose of contradiction, suppose that there is no $\left(S,b\right)$-addable element $(y,b)\notin F$, and that there are fewer than $n-\left|F_{b}\right|$ colours which are $\left(S,b\right)$-swappable. Let $S'\subseteq S$ be the set of all elements of
$S$ which have an $\left(S,b\right)$-swappable colour, so $\left|S'\right|<n-\left|F_{b}\right|$.
Also $\left|S'\right|<\left|S\right|$ because otherwise we would have $|S|<n-\left|F_{b}\right|$, so by the augmentation property there would be $y \in B_b \setminus F_b$ such that $S+(y,b)$ is an RIS (meaning that $(y,b)\notin F$ would be $\left(S,b\right)$-addable). Repeating this argument for $S'$ in place of $S$, there is $y\in B_{b}\backslash F_{b}$
such that $S'+\left(y,b\right)$ is an RIS. By repeatedly using the augmentation
property, we can add $\left|S-S'\right|-1$ elements of $S-S'$ to
$S'+\left(y,b\right)$. This gives an RIS of size $|S|$ of the form $S+\left(y,b\right)-\left(x',c\right)$
for some $\left(x',c\right)\in S-S'$. But this means $c$ is $\left(S,b\right)$-swappable, so $(x',c) \in S'$ by the definition of $S'$. This is a contradiction.
\end{proof}
Now we show that all elements of an $\left(S,b\right)$-swappable
colour which are independent to $\pi\left(S\right)$ are $\left(S,b\right)$-addable,
unless there is an \emph{$\left(S,b\right)$}-addable element not
in $F$. (Recall that $\pi\left(S\right)$ is the set of matroid elements
in $S$, without colour data.)
\begin{claim}
\label{claim:add-if-good}Consider an RIS $S$ with no element of
a colour $b$ and consider a colour
$c$ that is $\left(S,b\right)$-swappable with witness $y$. Either
$S+\left(y,b\right)$ is an RIS (thus, $\left(y,b\right)\notin F$
is $\left(S,b\right)$-addable), or otherwise for any $x\in B_{c}$
independent of $\pi\left(S\right)$, $\left(x,c\right)$ is $\left(S,b\right)$-addable with witness $(y,b)$.
\end{claim}

\begin{proof}
Let $\left(x',c\right)$ be the element with colour $c$ in $S$.
Consider some $x\in B_{c}$ independent to $\pi\left(S\right)$. Let $I=\pi\left(S\right)+x$ and $J=\pi\left(S\right)+y-x'$. By the augmentation property, there is an element of $I\backslash J$ that is independent of $J$; this element is either $x'$ or $x$. In the former case $S+\left(y,b\right)$
is an RIS. In the latter case, $S+\left(y,b\right)-\left(x',c\right)+\left(x,c\right)$ is
an RIS, showing that $\left(x,c\right)$ is $\left(S,b\right)$-addable. 
\end{proof}

The following lemma gives a good illustration of how to use the ideas developed in this section to find many addable elements. It will be very useful later on.

\begin{claim}\label{claim:1-addability} 
Let $S\in\S$ and let $b$ be a colour which does not appear in $S$. Then either we can increase the volume of $\S$ or there are at least $(n-|S|)\left(n-f\right)$ elements that are $\left(S,b\right)$-addable.
\end{claim}
\begin{proof}
If there is an element $\left(y,b\right)\notin F$ which is $\left(S,b\right)$-addable,
then we can directly add this element to $S$ (making a simple swap if necessary), increasing the volume
of $\S$. Otherwise, observe that $\left|F_{b}\right|\le\left|\S\right|=f$,
so by \cref{claim:many-good} there are at least $n-f$ colours that
are $\left(S,b\right)$-swappable. For
each such colour $c$, by the augmentation property, there are at least $n-|S|$ elements $x\in B_{c}$
independent to all the elements of $S$, each of which is $\left(S,b\right)$-addable
by \cref{claim:add-if-good}. That is to say, there are at least $(n-|S|)\left(n-f\right)$
elements which are $\left(S,b\right)$-addable, as claimed.
\end{proof}

In our proof of \cref{thm:new} we also make use of the following lemma. In the course of our arguments, when we need to find many addable elements with a given colour, it will allow us to ensure that these elements are actually distinct.

\begin{lem}
\label{lem:matching}Let $S$ be an RIS. Then for each $B_{b}$, we
can find an injection $\phi_{b}:S\to B_{b}$ such that for all
$\left(x,c\right)\in S$, $\phi_{b}\left(\left(x,c\right)\right)$
is independent of $\pi\left(S-\left(x,c\right)\right)$.
\end{lem}

\begin{proof}
Consider the bipartite graph $G$ where the first part consists of the elements of $S$ and
the second part consists of the elements of $B_{b}$, with an edge
between $\left(x,c\right)\in S$ and $y\in B_{b}$ if $y$ is independent
of $\pi\left(S-\left(x,c\right)\right)$. We use Hall's theorem
to show that there is a matching in this bipartite graph covering $S$. Indeed,
consider some $W\subseteq S$. By the augmentation property, there
are at least $\left|W\right|$ elements $y\in B_{b}$ such that $\pi\left(S-W\right)+y$
is an independent set, and again using the augmentation property,
each of these can be extended to an independent set of the form $\pi\left(S\right)+y-x$
for some $\left(x,c\right)\in W$. That is to say, $W$ has at least
$\left|W\right|$ neighbours in $G$.
\end{proof}

We thank the anonymous referees for pointing out that \cref{lem:matching} also follows from a result due to Brualdi \cite{brualdi69}.

\subsection{Cascading swaps}

Informally speaking, for any $S_{0}\in\S$ which is not a transversal
basis, we have shown that either we can directly augment $S_{0}$,
or there are many elements $\left(x_{1},c_{1}\right)\in U$ with which
we can augment $S_{0}$ after performing a simple swap. It's possible
that each such $\left(x_{1},c_{1}\right)$ already appears in some
other $S_{1}\in\S$, but if this occurs we need not give up: we can
transfer $\left(x_{1},c_{1}\right)$ from $S_{1}$ to $S_{0}$ and
then continue to look for elements $\left(x_{2},c_{2}\right)\in U$
with which we can augment $S_{1}-\left(x_{1},c_{1}\right)$ (again,
possibly with a swap). We can iterate this idea, looking for sequences
\[
S_{1},\dots,S_{\ell}\in\S,\quad\left(x_{1},c_{1}\right)\in S_{1},\,\left(x_{2},c_{2}\right)\in S_{2},\dots,\left(x_{\ell},c_{\ell}\right)\in S_{\ell},\,\left(x_{\ell+1},c_{\ell+1}\right)\notin\bigcup_{S\in\S}S
\]
such that, after a sequence of simple swaps, each $\left(x_{i},c_{i}\right)$
is transferred from $S_{i}$ to $S_{i-1}$, and then $\left(x_{\ell+1},c_{\ell+1}\right)$
can be added to $S_{\ell}$. (We also need to ensure that the simple swaps we perform preserve disjointness of RISs in $\S$.) This transformation has the net effect
of adding an element to $S_{0}$ and keeping the size of all other
$S\in\S$ constant, thus increasing the volume of $\S$.

Crucially, because of the freedom afforded by simple swaps, each time
we expand our search to consider longer cascades, our number of options
for $\left(x_{\ell+1},c_{\ell+1}\right)$ increases. For sufficiently
large $\ell$, the number of options will be so great that there must
be suitable $\left(x_{\ell+1},c_{\ell+1}\right)$ not appearing in
any RIS in $\S$. In order to keep this analysis tractable, we will
only consider transformations that cascade along a single sequence
of RISs $S_0,\dots,S_{\ell}$; we will iteratively construct this
sequence of RISs in such a way that there are many possibilities $\left(x_{i},c_{i}\right)\in S_{i}$
relative to the number of possibilities $\left(x_{i-1},c_{i-1}\right)\in S_{i-1}$
in the previous step. The next definition makes precise the cascades
that we consider.
\begin{defn}\label{Defn_Cascade}
Consider a sequence of distinct RISs $S_{0},\dots,S_{\ell-1}\in\S$.
Say an element $\left(x_{\ell},c_{\ell}\right)\notin S_{0},\dots,S_{\ell-1}$
is \emph{cascade-addable with respect to} $S_{0},\dots,S_{\ell-1}$
if there is a colour $c_{0}$ and sequences
\[
\left(x_{1},c_{1}\right),\dots,\left(x_{\ell-1},c_{\ell-1}\right)\in U,\qquad y_{0}\in B_{c_{0}},\dots,y_{\ell-1}\in B_{c_{\ell-1}},
\]
such that the following hold.
\begin{itemize}
\item For each $1\le i\le\ell-1$, we have $\left(x_{i},c_{i}\right)\in S_{i}$;
\item $c_{0}$ does not appear in $S_{0}$, and $\left(x_{1},c_{1}\right)$
is $\left(S_{0},c_{0}\right)$-addable with
witness $y_{0}$;
\item for each $0\le i\le\ell-1$, $\left(x_{i+1},c_{i+1}\right)$ is $\left(S_{i}-\left(x_{i},c_{i}\right),c_{i}\right)$-addable
with witness $y_{i}$;
\item the colours $c_0,\dots,c_\ell$ are distinct.
\end{itemize}
We call $c_0, c_1, \dots, c_{\ell-1}$  \emph{a sequence of colours freeing $(x_{\ell}, c_{\ell})$}.

We write $Q\left(S_{0},\dots,S_{\ell-1}\right)$ for the set of all
elements outside $S_{0},\dots,S_{\ell-1}$ which are cascade-addable
with respect to $S_{0},\dots,S_{\ell-1}$.
\end{defn}

We remark that if $\ell=1$ then most of the conditions in the above definition become vacuous and an element being cascade-addable with respect to $S_0$ is equivalent to it being $(S_0, c_0)$-addable with a witness, for some colour $c_0.$ 
Observe
that if an element $\left(x_{\ell},c_{\ell}\right)$ is cascade-addable
then we can transfer it into $S_{\ell-1}$, as the final step in a
cascading sequence of simple swaps and transfers. The following lemma makes this precise.

\begin{claim}\label{Claim_Perform_Cascade}
Suppose that $\left(x_{\ell},c_{\ell}\right)$
is {cascade-addable }with respect to $S_{0},\dots,S_{\ell-1}$ and $c_0, c_1, \dots, c_{\ell-1}$  is a sequence of colours freeing $(x_{\ell}, c_{\ell})$.  
Then there are $S_0'\dots S'_{\ell-1} \subseteq S_0\cup \dots \cup S_{\ell-1}\cup B_{c_0}\cup \dots \cup B_{c_{\ell-1}}$ such that replacing $S_0, \dots, S_{\ell-1}$  with $S'_0, \dots, S'_{\ell-1}$ in $\mathcal S$ results in a family $\mathcal S'$ of disjoint RISs of the same total volume as $\mathcal S$, in such a way that $S'_{\ell-1}+\left(x_{\ell},c_{\ell}\right)$ is an RIS.
\end{claim}
\begin{proof}
Let $\left(x_{1},c_{1}\right),\dots,\left(x_{\ell-1},c_{\ell-1}\right)\in U, y_{0}\in B_{c_{0}},\dots,y_{\ell-1}\in B_{c_{\ell-1}}$ be as in the definition of cascade-addability. For each $i=0, \dots, \ell-1$, let $(x'_i, c_{i+1})$ be the colour $c_{i+1}$ element of $S_i$ (which exists, because, from cascade-addability, $\left(x_{i+1},c_{i+1}\right)$ is $\left(S_{i}-\left(x_{i},c_{i}\right),c_{i}\right)$-addable
\emph{with a witness}). For each $i=1, \dots, \ell-2$, let $S_i'= S_i-(x_i,c_i)- (x'_i, c_{i+1})+ (y_i,c_i) + (x_{i+1}, c_{i+1})$.  Let $S_0'=S_0- (x'_0, c_{1}) +(y_0,c_0) + (x_{1}, c_{1})$ and $S_{\ell-1}'= S_{\ell-1}-(x_{\ell-1},c_{\ell-1})- (x'_{\ell-1}, c_{\ell})+ (y_{\ell-1},c_{\ell-1})$.  
Let $\mathcal S'$ be the family formed by replacing $S_0, \dots, S_{\ell-1}$  with $S'_0, \dots, S'_{\ell-1}$ in $\mathcal S$. It is easy to check that $\mathcal S'$ has the same total volume as $\mathcal S$, so it remains to check that it is a family of disjoint RISs.

For $i=1, \dots, \ell-2$, $S_i'$ is an RIS because it comes from  $S_i-(x_i,c_i)$ by making the change in the definition of $(x_{i+1}, c_{i+1})$ being $(S_i-(x_i,c_i), c_i)$-addable with witness $y_i$ (and addability always produces an RIS by definition). 
Similarly $S_{\ell-1}'+\left(x_{\ell},c_{\ell}\right)$ is an RIS.
To see that $S_0'$ is an RIS we use that $(x_1, c_1)$ is $(S_0,c_0)$-addable with witness $y_0$, and that $c_0$ does not appear in $S_0$, both of which come from the definition of cascade-addability.

It remains to show that the RISs $S_0', \dots, S_{\ell-1}'$ are disjoint from each other and the other RISs in $\mathcal S$. The elements $(y_i,c_i)$ occur in only one RIS  $S_i'$ because they come from outside  $F$ (since they are addability witnesses), and because their colours $c_0, \dots, c_{\ell-1}$ are distinct (from the definition of cascade-addability). The elements $(x_i, c_i)$ occur in only one RIS because they get removed from $S_{i}$ and added to $S_{i-1}$.
\end{proof}

The following lemma lets us build longer cascades.
\begin{claim}\label{Claim_Concatenate_cascade}
Suppose that $\left(x_{\ell},c_{\ell}\right)\in S_{\ell}$
is {cascade-addable }with respect to $S_{0},\dots,S_{\ell-1}$ and $c_0, c_1, \dots, c_{\ell-1}$  is a sequence of colours freeing $(x_{\ell}, c_{\ell})$. If $(x,c)$ is $(S_{\ell}-(x_{\ell}, c_{\ell}), c_{\ell})$-addable with a witness then either $(x,c)\in S_{0}\cup \dots\cup S_{\ell}\cup B_{c_0}\cup \dots\cup B_{c_{\ell}}$ or $(x,c)$ is {cascade-addable} with respect to $S_{0},\dots,S_{\ell}$.
\end{claim}
\begin{proof}
Suppose that $(x,c)\not \in S_{0},\dots,S_{\ell}, B_{c_0}, \dots, B_{c_{\ell}}$. 
For the definition of $(x,c)$ being {cascade-addable}, all the conditions not involving $(x,c)$  and $(x_{\ell},c_{\ell})$ hold as a consequence of  $\left(x_{\ell},c_{\ell}\right)\in S_{\ell}$
being {cascade-addable }with respect to $S_{0},\dots,S_{\ell-1}$.
It remains to check the conditions that $(x,c)\not\in S_{0},\dots,S_{\ell}$ and that
each of $c_0,\dots,c_\ell,c$ are distinct, both of which hold as a consequence of our assumption $(x,c)\not \in S_{0},\dots,S_{\ell}, B_{c_0}, \dots, B_{c_{\ell}}$.
\end{proof}

In the next lemma, we essentially show that given $S_{0},\dots,S_{\ell-1}$,
it is possible to choose $S_{\ell}$ in such a way that the number
of cascade-addable elements increases.
\begin{claim}
\label{claim:cascade-increase}Consider a sequence of distinct RISs
$S_{0},\dots,S_{\ell-1}\in\S$ with $1 \le \ell<f=\left|\S\right|$.
Then either we can modify $\S$ to increase its volume, or we can
choose $S_{\ell}\ne S_{0},\dots,S_{\ell-1}$ from $\S$ such that
\begin{equation}
\left|Q\left(S_{0},\dots,S_{\ell}\right)\right|\ge\frac{\left|Q\left(S_{0},\dots,S_{\ell-1}\right)\right|}{f-\ell}\cdot\left(n-f-\ell\right)-(\ell+1) n.\label{eq:recurrence}
\end{equation}
\end{claim}

\begin{proof}
If $Q\left(S_{0},\dots,S_{\ell-1}\right)$ contains an element $(x,c)$ not
in any $S\in\S$, then we can increase the volume of $\S$ with a
cascading sequence of simple swaps and transfers (using \cref{Claim_Perform_Cascade}, noting that if $\left(x_{\ell},c_{\ell}\right)\not\in F$, then we can add $\left(x_{\ell},c_{\ell}\right)$ to $S'_{\ell-1}$ in that lemma to get a larger family of RISs).

Otherwise, all the elements of $Q\left(S_{0},\dots,S_{\ell-1}\right)$ belong to some RIS $S\in \S\setminus\left\{ S_{0},\dots S_{\ell-1}\right\}$ (since $Q\left(S_{0},\dots,S_{\ell-1}\right)$ is defined to not contain any elements from $S_{0},\dots,S_{\ell-1}$).
Choose $S_{\ell}\in\S\setminus\left\{ S_{0},\dots S_{\ell-1}\right\}$ containing maximally many elements of $Q\left(S_{0},\dots,S_{\ell-1}\right)$. Since the $f-\ell$ RISs $S\in \S\setminus\left\{ S_{0},\dots S_{\ell-1}\right\}$ collectively contain all elements of $Q\left(S_{0},\dots,S_{\ell-1}\right)$, our chosen RIS $S_\ell$ must contain a proportion of at least $1/(f-\ell)$ of the elements of $Q\left(S_{0},\dots,S_{\ell-1}\right)$. In other words, if we let $Q=S_{\ell}\cap Q\left(S_{0},\dots,S_{\ell-1}\right)$, we have
\begin{equation}\label{eq:QBound}
\left|Q\right|\ge\frac{\left|Q\left(S_{0},\dots,S_{\ell-1}\right)\right|}{f-\ell}.
\end{equation}
Apply \cref{lem:matching} to $S_{\ell}$ to obtain an injection $\phi_{b}$, for every colour $b$.
Fix some $(x_{\ell},c_{\ell})\in Q$ and a sequence of colours $c_0, \dots, c_{\ell-1}$ freeing $(x_{\ell},c_{\ell})$. We prove a sequence of claims about how many elements are swappable/addable with respect to $\left(S_{\ell}-\left(x_{\ell},c_{\ell}\right),c_{\ell}\right)$, assuming we cannot increase the size of $\S$.
\begin{claim*}
{There are at least $n-f$ colours
which are $\left(S_{\ell}-\left(x_{\ell},c_{\ell}\right),c_{\ell}\right)$-swappable.}
\end{claim*}
\begin{proof}
By \cref{claim:many-good}, either there is an $\left(S_{\ell}-\left(x_{\ell},c_{\ell}\right),c_{\ell}\right)$-addable element $(y,c_\ell)\not\in F$, or there are at least $n-\left|F_{c_{\ell}}\right|\ge n-f$ colours
which are $\left(S_{\ell}-\left(x_{\ell},c_{\ell}\right),c_{\ell}\right)$-swappable. In the former case, we can increase the volume of $\S$, by a cascading sequence of swaps and transfers (first consider $\mathcal S'$ from \cref{Claim_Perform_Cascade}, then move $(x_\ell,c_\ell)$ from $S_{\ell}$ to $S_{\ell-1}'$, then add $(y,c_\ell)$ to $S_\ell-(x_\ell,c_\ell)$).
\end{proof}
\begin{claim*}
{There are at least $n-f$ colours $c$ for which $\left(\phi_{c}\left(\left(x_{\ell},c_{\ell}\right)\right),c\right)$
is $\left(S_{\ell}-\left(x_{\ell},c_{\ell}\right),c_{\ell}\right)$-addable.}
\end{claim*}
\begin{proof}
Let $c$ be a colour  which is $\left(S_{\ell}-\left(x_{\ell},c_{\ell}\right),c_{\ell}\right)$-swappable with witness $y$, as in the previous claim. If $y$ is independent
to $\pi\left(S_{\ell}-\left(x_{\ell},c_{\ell}\right)\right)$, we
can increase the volume of $\S$ by adding it to $S_{\ell}$ after
a cascading sequence of swaps and transfers (first consider $\mathcal S'$ from \cref{Claim_Perform_Cascade}, then move $(x_\ell,c_\ell)$ from $S_{\ell}$ to $S_{\ell-1}'$, then add $(y,c_\ell)$ to $S_\ell-(x_\ell,c_\ell)$). 
Otherwise, by \cref{claim:add-if-good} applied with $b=c_{\ell}$, $S=S_{\ell}-\left(x_{\ell},c_{\ell}\right)$, the element
$\left(\phi_{c}\left(\left(x_{\ell},c_{\ell}\right)\right),c\right)$
is $\left(S_{\ell}-\left(x_{\ell},c_{\ell}\right),c_{\ell}\right)$-addable. Here we are using that $\left(\phi_{c}\left(\left(x_{\ell},c_{\ell}\right)\right),c\right)$
is independent from $\left(S_{\ell}-\left(x_{\ell},c_{\ell}\right),c_{\ell}\right)$ (which comes from the definition of $\phi_c$ in \cref{lem:matching}).
\end{proof}
\begin{claim*}
{There are at least $n-f-\ell$ colours $c\not\in\{c_0, \dots, c_{\ell-1}\}$ for which $\left(\phi_{c}\left(\left(x_{\ell},c_{\ell}\right)\right),c\right)$
is $\left(S_{\ell}-\left(x_{\ell},c_{\ell}\right),c_{\ell}\right)$-addable.}
\end{claim*}
\begin{proof}
This ensues from the previous claim and the fact that the only requirement on $c$, besides addability, is that it is different from the $\ell$ colours in $\{c_0, \dots, c_{\ell-1}\}$.
\end{proof}
We now prove the following:
\begin{equation}\label{eq:QBound2}
|Q\left(S_{0},\dots,S_{\ell}\right)|\geq  \left|Q\right|\left(n-\ell-f\right)-(\ell+1) n.
\end{equation}
From the last claim, we have $\left|Q\right|\left(n-\ell-f\right)$ elements of the form $\left(\phi_{c}\left(\left(x_{\ell},c_{\ell}\right)\right),c\right)$ which are all $\left(S_{\ell}-\left(x_{\ell},c_{\ell}\right),c_{\ell}\right)$-addable, with $c$  outside a sequence of colours freeing $(x_{\ell}, c_{\ell})$. Notice that these $\left(\phi_{c}\left(\left(x_{\ell},c_{\ell}\right)\right),c\right)$ are all distinct because $\phi_c$ is an injection.
By \cref{Claim_Concatenate_cascade}, each of these is cascade-addable with respect to $S_{0},\dots,S_{\ell}$,  unless it appears in one of $S_{0},\dots,S_{\ell}$.
The total number of elements in $S_{0},\dots,S_{\ell}$ is at most $(\ell+1) n$, so we have found $\left|Q\right|\left(n-\ell-f\right)-(\ell+1) n$ cascade-addable elements with respect to $S_{0},\dots,S_{\ell}$, as required by \cref{eq:QBound2}. 

The lemma immediately follows by combining \cref{eq:QBound} and \cref{eq:QBound2}.
\end{proof}

Now, we want to iteratively apply \cref{claim:cascade-increase} starting
from some $S_{0}\in\S$, to obtain a sequence $S_{0},S_{1},\dots,S_{h}\in\S$.
There are two ways this process can stop: either we find a way to
increase the volume of $\S$, in which case we are done, or else we
run out of RISs in $\S$ (that is, $h=f-1$). We want to show that
this latter possibility cannot occur by deducing from \cref{eq:recurrence}
that the $\left|Q\left(S_{0},\dots,S_{\ell}\right)\right|$ increase
in size at an exponential rate: after logarithmically many steps there
will be so many cascade-addable elements that they cannot all be contained
in the RISs in $\S$, and it must be possible to increase the volume
of $\S$.

A slight snag with this plan is that \cref{eq:recurrence} only yields
an exponentially growing recurrence if the ``initial term'' is rather
large. To be precise, let $C$ (depending on $\varepsilon$) be sufficiently
large such that 
\begin{equation}
C\left(1+\varepsilon/2\right)^{\ell-1}\frac{1}{1-\varepsilon}-\ell-1\ge C\left(1+\varepsilon/2\right)^{\ell}\label{eq:C}
\end{equation}
for all $\ell\ge1$.
\begin{claim}
\label{claim:recurrence-estimate}For $S_{0},\dots,S_{h}$ as above,
suppose that $\left|Q\left(S_{0}\right)\right|\ge Cn$ or $\left|Q\left(S_{0},S_{1}\right)\right|\ge Cn$.
Then, for $0<\ell\le\min\left\{ h,\varepsilon n/2\right\} $, we have
\[
\left|Q\left(S_{0},\dots,S_{\ell}\right)\right|\ge C\left(1+\varepsilon/2\right)^{\ell-1}n.
\]
\end{claim}

\begin{proof}
We first establish a technical inequality. Recall that $f=(1-\varepsilon)n/2,$ so 
\begin{equation}
\frac{n-f-\ell}{f-\ell} \ge \frac{n-(1-\varepsilon)n/2-n\varepsilon/2}{(1-\varepsilon)n/2}= \frac{1}{1-\varepsilon}.\label{eq:fneps}
\end{equation}
Now, let $Q_{\ell}=Q\left(S_{0},\dots,S_{\ell}\right)$. We proceed by
induction. First observe that if $|Q_0| \ge Cn$ then \cref{eq:recurrence}, \cref{eq:fneps} and \cref{eq:C} for $\ell=1$ imply $|Q_1|\ge Cn(n-f-1)/(f-1)-2n \ge(C/(1-\varepsilon)-2)n\ge Cn$, giving us the base case. If $\left|Q_{\ell}\right|\ge C\left(1+\varepsilon/2\right)^{\ell-1}n$, 
then once again using \cref{eq:recurrence}, \cref{eq:fneps} and \cref{eq:C}, we obtain
\begin{align*}
\left|Q_{\ell+1}\right| & \ge\frac{C\left(1+\varepsilon/2\right)^{\ell-1}n}{f-\ell}\cdot\left(n-f-\ell\right)-(\ell+1) n\\
 & =\left(C\left(1+\varepsilon/2\right)^{\ell-1}\frac{\left(n-f-\ell\right)}{f-\ell}-\ell-1\right)n\\
 & \ge \left(C\left(1+\varepsilon/2\right)^{\ell-1}\frac{1}{1-\varepsilon}-\ell-1\right)n\\
 & \ge C\left(1+\varepsilon/2\right)^{\ell}n.\tag*{\qedhere}
\end{align*}
\end{proof}
If we could choose $S_{0},S_{1}$ such that $\left|Q\left(S_{0}\right)\right|\ge Cn$
or $\left|Q\left(S_{0},S_{1}\right)\right|\ge Cn$, then \cref{claim:recurrence-estimate}
would imply that during the construction of $S_1,\dots,S_h$ we never run out of RISs in $\S$ (that is, $h<f-1$). Indeed, otherwise $Q(S_0,\dots,S_{\varepsilon n/2})$ would have size exponential in $n,$ which is impossible. Therefore, the process must stop at some point when we find a way to increase the volume of $\S.$ Provided we can again find suitable $S_0,S_1$ we can then repeat the arguments in this section, further increasing the volume of $\S$. After repeating these arguments enough times we will have obtained
$f=\left(1-\varepsilon\right)n/2\ge\left(1/2-\varepsilon\right)n$
disjoint transversal bases, completing the proof of \cref{thm:new}.

There may not exist suitable $S_{0},S_{1}\in\S$, but in the next
section we will show that if at least $\varepsilon n/2$ of the RISs
in $S$ are not transversal bases, then it is possible to modify $\S$
without changing its volume, in such a way that suitable $S_{0},S_{1}$
exist.
\begin{rem}
\label{rem:linear}With the results we have proved so far, we can already find linearly many disjoint transversal bases. Indeed, if $S_{0}$ is not a transversal basis (missing
a colour $b$, say), and the volume of $\S$ cannot be increased by
adding an element to $S_{0}$ (possibly after a simple swap), then
\cref{claim:1-addability} implies that there are at
least $n-f$ elements which are $\left(S_{0},b\right)$-addable,
meaning that $\left|Q\left(S_{0}\right)\right|\ge n-f$.
Take for example $\varepsilon=4/5$, meaning that $f\le n/10$ and $\left|Q\left(S_{0}\right)\right|\ge9n/10$. We can check
that \cref{eq:C} holds for all $\ell\ge1$ if $C=9/10$. That is to
say, as long as we have not yet completed $\S$ to a collection of
disjoint transversal bases, we can keep increasing its volume without
the considerations in the next section. This proves already
that it is possible to find linearly many disjoint transversal bases.
\end{rem}
\begin{rem}
It is not hard to add a term $(n-|S_{\ell}|)(n-f)$ to the right hand side of the inequality given by \cref{claim:cascade-increase} by considering also cascades along the sequence $S_0, \ldots, S_{\ell-1}$ of length strictly less than $\ell$. However, since this increase is only significant when $|S_{\ell}|$ is not close to $n$, which may never be the case, we omit it from our argument for the sake of readability.
\end{rem}

\subsection{Increasing the number of initial addable elements}\label{subsec:increasing}

Consider a collection $\S$ of $f=\left(1-\varepsilon\right)n/2$
disjoint RISs, at least $\varepsilon n/2$ of which are not transversal
bases. Recall the choice of $C$ from the previous section, and let $D=2C+4$, so that
$D\left(n-f-1\right)-2n\ge Cn$ for large $n$. We prove the following (for large $n$).

\begin{claim}
\label{claim:many-missing}We can modify $\S$ in such a way that at least one of the following holds.
\begin{enumerate}
\item [(a)]The volume of $\S$ increases;
\item [(b)]the volume of $\S$ does not change, and there is now $S_{0}\in\S$
missing at least $D$ colours;
\item [(c)]the volume of $\S$ does not change, and there are now distinct
$S_{0},S_{1}\in\S$ such that $S_{1}$ contains at least $D$ elements
that are $\left(S_{0},b\right)$-addable, for some colour $b$.
\end{enumerate}
\end{claim}

This suffices for our proof of \cref{thm:new}; indeed, if $S_{0}$ is missing at least $D$ colours, then by \cref{claim:1-addability}, either we can increase the volume of $\S$
or there are at least $D\left(n-f\right)\ge Cn$ elements which are
$\left(S_{0},b\right)$-addable for every $b$ not appearing in $S_0$, meaning that $\left|Q\left(S_{0}\right)\right|\ge Cn$.
If $S_{1}$ contains at least $D$ elements that are $\left(S_{0},b\right)$-addable,
then in the proof of \cref{claim:cascade-increase} with $\ell=1$ we have $|Q|\ge D$ so either we can increase the volume of $\S$ or $\left|Q\left(S_{0},S_{1}\right)\right|\ge D\left(n-f-1\right)-2n\ge Cn$ (recall \cref{eq:QBound2}).

Before proceeding to the proof of \cref{claim:many-missing},
we first observe that using \cref{claim:many-good} we can modify $\S$
to ensure that every $S\in\S$ that is not a transversal basis can
be assigned a distinct missing colour $b\left(S\right)$. To see this,
we iteratively apply the following lemma to $\S$.
\begin{lem}
\label{lem:different-missing}
Consider $f\le n/2$ and let $\S=\left\{ S_{1},\dots,S_{f}\right\} $
be a collection of disjoint RISs. We can either increase the size of $\S$ or we can modify $\S$ in such a way that the size of each $S_i$ remains the same, and in such a way that that there is a choice of disjoint colours $\{ b_1,\dots,b_f\}$ for which any $S_i$ that is not a transversal basis has no element of colour $b_i.$
\end{lem}
\begin{proof}
Suppose for some $i$ that we found distinct colours $b_1,\dots,b_{i-1}$
such that, for all $S_{j}$ which are not transversal bases, no element
of $S_{j}$ is of colour $b_{j}$. If $S_i$ is a transversal basis we choose an arbitrary unused colour as $b_i.$ Otherwise there is a colour, say $c$, not appearing in $S_{i}$. Then by \cref{claim:many-good} either we can increase the size of $\S$ or there are at least $n-f\ge n/2$ colours which are $(S_i,c)$-swappable. At least one of these colours does not appear in $\left\{ b_1,\dots,b_{i-1}\right\}$, since $i-1<f\le n/2$. Let $b$ be such a colour and set $b_i=b$. By performing a simple swap, we transform $S_i$ into a new RIS, still disjoint to all other $S_j \in \S$ and missing the colour $b$.
\end{proof}

Now we prove \cref{claim:many-missing}.
\begin{proof}[Proof of \cref{claim:many-missing}]
Recall that we are assuming there are at least $\varepsilon n/2$
RISs in $\S$ that are not transversal bases. Let $E$ be the largest
integer such that there are at least $M_{E}=\left(\varepsilon/\left(4D^{2}\right)\right)^{E}n$
RISs in $\S$ missing at least $E$ colours. We may assume $1\le E<D$.
By \cref{lem:different-missing} we may assume that each $S\in\S$ which is not a transversal basis
has a distinct missing colour $b\left(S\right)$. We describe a procedure
that modifies $\S$ to increase $E$.

We create an auxiliary digraph $G$ on the vertex set $\S$ as follows.
For every $S_{0}\in\S$ missing at
least $E$ colours, put an arc to $S_{0}$ from every $S_{1}\in\S$
such that $S_{1}$ contains at least $E+1$ elements that are $\left(S_{0},b\left(S_{0}\right)\right)$-addable.

Say an \emph{$\left(E+1\right)$-out-star} in a digraph is a set of
$E+1$ arcs directed away from a single vertex. Our goal is to prove that there are $M_{E+1}$ vertex-disjoint $\left(E+1\right)$-out-stars. To see why this suffices, consider an $\left(E+1\right)$-out-star (with centre $S_{1}$, say). We show how to transfer $E+1$ elements from $S_1$ to its out-neighbours, the end result of which is that $S_1$ is then missing $E+1$ colours. We will then be able to repeat this process for each of our out-stars.

For each of the $E+1$ out-neighbours $S_{0}$ of $S_{1}$ there are
at least $E+1$ elements of $S_1$ which are $\left(S_{0},b\left(S_{0}\right)\right)$-addable. Therefore, for each such $S_0$ we can make a specific choice of such an $\left(S_{0},b\left(S_{0}\right)\right)$-addable element, in such a way that each of these $E+1$ choices are \emph{distinct}. For each $S_0$ we can then transfer the chosen element from $S_1$ to $S_0$, possibly with a simple
swap. These simple swaps will not create any conflicts, because
any addability witness for any element in $S_{0}$ is in a colour
unique to that $S_{0}$ (by the property from \cref{lem:different-missing}). After this
operation, $S_i$ is now missing
at least $E+1$ colours.

It will be a relatively straightforward matter to find our desired out-stars by studying the digraph $G$. First we show that $G$ must have many edges.
\begin{claim*}
In the above auxiliary digraph, we may assume that every $S_{0}\in\S$ missing at least
$E$ colours has in-degree at least $\varepsilon n/D$.
\end{claim*}
\begin{proof}
By \cref{claim:1-addability} we can
assume that there are at least $E\left(n-f\right)$ elements which are $\left(S_{0},b\left(S_{0}\right)\right)$-addable. All these elements appear in various $S\in\S$ (otherwise we can increase the volume of $\S$).

Let $N^-(S_0)$ be the set of all $S_1$ such that there is an arc from $S_1$ to $S_0$ in $G$ (so $|N^-(S_0)|$ is the indegree of $S_0$). By definition, every $S\notin N^-(S_0)$ has at most $E$ elements which are $\left(S_{0},b\left(S_{0}\right)\right)$-addable. Moreover, observe that every $S\in \S$ has fewer than $D$ elements that are $\left(S_{0},b(S_0)\right)$-addable, or else (c) trivially occurs. It follows that
\begin{align*}
D|N^-(S_0)|+E(f-|N^-(S_0)|)&\ge E(n-f),
\end{align*}
so
\begin{align*}
|N^-(S_0)| & \ge\frac{E\left(\left(n-f\right)-f\right)}{D-E}\ge\frac{\varepsilon n}{D},
\end{align*}
as desired.
\end{proof}
We have proved that $G$ has at least $M_E \varepsilon n/D$ edges. Now we finish the proof by showing how to find our desired out-stars.
\begin{claim*}
$G$ has at least $M_{E+1}$ vertex-disjoint $\left(E+1\right)$-out-stars.
\end{claim*}
\begin{proof}
We can find these out-stars in a greedy fashion. Suppose that we have already found $t$ vertex-disjoint $\left(E+1\right)$-out-stars, for some $t< M_{E+1}$. We show that there must be an additional $\left(E+1\right)$-out-star disjoint to these. Let $G'$ be obtained from $G$ by deleting all vertices in the out-stars we have found so far. Each of these out-stars has $E+2$ vertices, so the number of arcs in $G'$ is at least
\begin{align*}
M_E \frac{\varepsilon n}D-t(E+2)\cdot2f&>M_E \frac{\varepsilon n}D-M_{E+1}(E+2)\cdot 2f \\
& = M_E \frac{\varepsilon n}D-\frac{M_{E}\varepsilon}{2D^2}\cdot(E+2)f\\
&\ge M_E \varepsilon\left(\frac n D-\frac f{D}\right)\\
& \ge M_E \varepsilon \cdot \frac f{D}\ge (E+1)f,
\end{align*}
where the last inequality holds for sufficiently large $n$, using the fact that $M_E$ is linear in $n$. This means that $G'$ (having at most $f$ vertices) has a vertex with outdegree at least $E+1$, which means $G'$ contains an $\left(E+1\right)$-out-star disjoint to the out-stars we have found so far.
\end{proof}
\end{proof}

\section{Concluding remarks}

In this paper we proved that that given bases $B_{1},\dots,B_{n}$
in a matroid, we can find $\left(1/2-o\left(1\right)\right)n$ disjoint
transversal bases. Although our methods do not extend past $n/2$,
we do not think that there is a fundamental obstacle preventing related
methods from going further. Indeed, by tracking the possible cascades
of swaps more carefully, it might be possible to find $\left(1-o\left(1\right)\right)n$
disjoint transversal bases, or at least to find $\left(1-o\left(1\right)\right)n$
disjoint partial transversals each of size $\left(1-o\left(1\right)\right)n$.
Although we cannot completely rule out the possibility that a full
proof of Rota's basis conjecture could be obtained in this way, we
imagine that more ingredients will be required. We are hopeful that
ideas used to prove existence of designs (see \cite{Kee14,GKLO16})
could be relevant, at least in the case of vector spaces.

Also, we remark that Rota's basis conjecture is reminiscent of some
other problems concerning rainbow
structures in graphs (actually, for a graphic matroid, Rota's basis
conjecture can be interpreted as a conjecture about rainbow spanning
forests in edge-coloured multigraphs). The closest one to Rota's basis
conjecture seems to be the Brualdi--Hollingsworth conjecture
\cite{BH96}, which posits that for every $(n-1)$-edge-colouring of the complete graph $K_n$, the edges can be decomposed into rainbow spanning trees. This conjecture has
recently seen some exciting progress (see for example \cite{Hor18,PS18,BLM18,MPS18}).
We wonder if some of the ideas developed for the study of rainbow
structures could be profitably applied to Rota's basis conjecture.

We also mention the following strengthening of Rota's basis conjecture due to Kahn (see \cite{HR94}). This is simultaneously a strengthening
of the Dinitz conjecture \cite{Dinitz} on list-colouring of $K_{n,n}$,
solved by Galvin \cite{Gal95}.
\begin{conjecture}
\label{conj:kahn}Given a rank-$n$ matroid and bases $B_{i,j}$ for
each $1\le i,j\le n$, there exist representatives $b_{i,j}\in B_{i,j}$
such that each of the sets $\{b_{1,j},\dots,b_{n,j}\}$ and
$\{b_{i,1},\dots,b_{i,n}\}$ are bases.
\end{conjecture}

The methods developed in this paper are also suitable for studying  
\cref{conj:kahn}. In particular, the argument used to prove \cref{thm:new} can readily be modified to show the following natural partial result towards Kahn's conjecture.
\begin{thm}\label{thm:kahn}
For any $\varepsilon>0$ the following holds for sufficiently large $n$. Given a rank-$n$ matroid and bases $B_{i,j}$ for
each $1\le i\le n$ and $1\le j \le f=(1-\varepsilon)n/2$, there exist representatives $b_{i,j}\in B_{i,j}$ and $L\subseteq \{1,\dots,f\}$
such that each $\{b_{i,j}:i\in L\}$ is independent, and such that
$\{b_{i,1},\dots,b_{i,n}\}$ is a basis for any $i \in L$ and $|L| \ge (1/2-\varepsilon)n$.
\end{thm}

Note that if we are in the setting of \cref{conj:kahn} where bases are given for all $1 \le i,j\le n$ then the above theorem allows us to choose roughly which rows we would like to find our bases in. 

Note also that if, for each fixed $j$, the bases $B_{1,j},\dots,B_{n,j}$ are all equal, then Kahn's conjecture reduces to Rota's basis conjecture. This observation also shows that \cref{thm:kahn} implies \cref{thm:new}.

It is not hard to adapt the proof of \cref{thm:new} to prove \cref{thm:kahn}. However, since it would require repeating most of the argument, we omit the details here. For interested readers we present the details in a companion note, which we will not publish but will make available on the arXiv \cite{arxiv:kahn}.

\medskip
\textbf{Acknowledgements.} We are extremely grateful to the anonymous referees for their careful reading of the paper and many useful suggestions.


\begin{thebibliography}{10}

\bibitem{AB06}
R.~Aharoni and E.~Berger, \emph{The intersection of a matroid and a simplicial
  complex}, Trans. Amer. Math. Soc. \textbf{358} (2006), no.~11, 4895--4917.

\bibitem{AT92}
N.~Alon and M.~Tarsi, \emph{Colorings and orientations of graphs},
  Combinatorica \textbf{12} (1992), no.~2, 125--134.

\bibitem{BLM18}
J.~Balogh, H.~Liu and R.~Montgomery, \emph{Rainbow spanning trees in properly coloured complete graphs},
  Discrete Appl. Math. \textbf{247} (2018), no.~1, 97--101.
  
\bibitem{BD15}
G.~P. Bollen and J.~Draisma, \emph{An online version of {R}ota's basis
  conjecture}, J. Algebraic Combin. \textbf{41} (2015), no.~4, 1001--1012.

\bibitem{brualdi69}
R.~A. Brualdi, \emph{Comments on bases in dependence structures}, Bull. Austral. Math. Soc. \textbf{1} (1969), 161--167.


\bibitem{BH96}
R.~A. Brualdi and S.~Hollingsworth, \emph{Multicolored trees in complete
  graphs}, J. Combin. Theory Ser. B \textbf{68} (1996), no.~2, 310--313.

\bibitem{arxiv:kahn}
M.~Bucic, M.~Kwan, A.~Pokrovskiy, B.~Sudakov, \emph{On Kahn's basis conjecture}, companion note.

\bibitem{Che12}
M.~Cheung, \emph{Computational proof of {R}ota's basis conjecture for matroids
  of rank 4}, unpublished manuscript,
  \url{http://educ.jmu.edu/~duceyje/undergrad/2012/mike.pdf}, 2012.

\bibitem{polyproposal}
T.~Chow, \emph{Proposals for polymath projects},
  \url{https://mathoverflow.net/q/231153}, 2017.

\bibitem{poly1}
T.~Chow, \emph{Rota's basis conjecture: Polymath 12},
  \url{https://polymathprojects.org/2017/02/23/rotas-basis-conjecture-polymath-12/},
  2017.

\bibitem{poly2}
T.~Chow, \emph{Rota's basis conjecture: Polymath 12},
  \url{https://polymathprojects.org/2017/03/06/rotas-basis-conjecture-polymath-12-2/},
  2017.

\bibitem{poly3}
T.~Chow, \emph{Rota's basis conjecture: Polymath 12, post 3},
  \url{https://polymathprojects.org/2017/05/05/rotas-basis-conjecture-polymath-12-post-3/},
  2017.

\bibitem{polywiki}
T.~Chow, \emph{Rota's conjecture},
  \url{http://michaelnielsen.org/polymath1/index.php?title=Rota%27s_conjecture},
  2017.

\bibitem{DG18}
S.~Dong and J.~Geelen, \emph{Improved bounds for {R}ota's basis conjecture},
  Combinatorica \textbf{39} (2019), no.~2, 265--272. 

\bibitem{Dri97}
A.~A. Drisko, \emph{On the number of even and odd {L}atin squares of order
  {$p+1$}}, Adv. Math. \textbf{128} (1997), no.~1, 20--35.

\bibitem{Dinitz}
P.~Erd{\H o}s, A.~L. Rubin, and H.~Taylor, \emph{Choosability in graphs},
  Proceedings of the {W}est {C}oast {C}onference on {C}ombinatorics, {G}raph
  {T}heory and {C}omputing ({H}umboldt {S}tate {U}niv., {A}rcata, {C}alif.,
  1979), Congress. Numer., XXVI, Utilitas Math., Winnipeg, Man., 1980,
  pp.~125--157.

\bibitem{Gal95}
F.~Galvin, \emph{The list chromatic index of a bipartite multigraph}, J.
  Combin. Theory Ser. B \textbf{63} (1995), no.~1, 153--158.

\bibitem{GH06}
J.~Geelen and P.~J. Humphries, \emph{{R}ota's basis conjecture for paving
  matroids}, SIAM J. Discrete Math. \textbf{20} (2006), no.~4, 1042--1045.

\bibitem{GW07}
J.~Geelen and K.~Webb, \emph{On {R}ota's basis conjecture}, SIAM J. Discrete
  Math. \textbf{21} (2007), no.~3, 802--804.

\bibitem{GKLO16}
S.~Glock, D.~K{\"u}hn, A.~Lo, and D.~Osthus, \emph{The existence of designs via
  iterative absorption}, arXiv preprint arXiv:1401.3665 (2016).

\bibitem{Gly10}
D.~G. Glynn, \emph{The conjectures of {A}lon-{T}arsi and {R}ota in dimension
  prime minus one}, SIAM J. Discrete Math. \textbf{24} (2010), no.~2, 394--399.

\bibitem{Hor18}
P.~Horn, \emph{Rainbow spanning trees in complete graphs colored by
  one-factorizations}, J. Graph Theory \textbf{87} (2018), no.~3, 333--346.

\bibitem{HR94}
R.~Huang and G.-C. Rota, \emph{On the relations of various conjectures on
  {L}atin squares and straightening coefficients}, Discrete Math. \textbf{128}
  (1994), no.~1-3, 225--236.

\bibitem{Kee14}
P.~Keevash, \emph{The existence of designs}, arXiv preprint arXiv:1401.3665
  (2014).

\bibitem{MPS18}
R.~Montgomery, A.~Pokrovskiy, and B.~Sudakov, \emph{Decompositions into
  spanning rainbow structures}, Proc. Lond. Math. Soc. \textbf{119} (2019), no.~4, 899--959. 

\bibitem{Onn97}
S.~Onn, \emph{A colorful determinantal identity, a conjecture of {R}ota, and
  {L}atin squares}, Amer. Math. Monthly \textbf{104} (1997), no.~2, 156--159.

\bibitem{PS18}
A.~Pokrovskiy, and B.~Sudakov, \emph{Linearly many rainbow trees in properly edge-coloured complete graphs}, J. Combin. Theory Ser. B \textbf{132} (2018), no.~1, 134--156.


\bibitem{Pol17}
D.~H.~J. Polymath, \emph{Rota's basis conjecture online for matroids},
  unpublished manuscript, \url{https://www.overleaf.com/8773999gccdbdmfdgkm},
  2017.

\bibitem{Rad42}
R.~Rado, \emph{A theorem on independence relations}, Quart. J. Math., Oxford
  Ser. \textbf{13} (1942), 83--89.

\bibitem{Wil94}
M.~Wild, \emph{On {R}ota's problem about {$n$} bases in a rank {$n$} matroid},
  Adv. Math. \textbf{108} (1994), no.~2, 336--345.
\end{thebibliography}
\end{document}